\documentclass[11pt]{article}

\usepackage{amsmath,amsthm,verbatim,amssymb,amsfonts,amscd, graphicx}
\usepackage{graphics}
\usepackage[utf8]{inputenc}
\theoremstyle{plain}

\title{\textrm{The chromatic number of ($P_{5}$, \textit{HVN})-free graphs}\footnote{Supported by NSFC No. 12101117, and by NSFJS No. BK20200344}}

\author{Yian Xu\footnote{Email: yian$\_$xu@seu.edu.cn}\\School of Mathematics, Southeast University\\2 SEU Road, Nanjing, 211189, China}

\date{}

\newtheorem{example}{Example}[section]
\newtheorem{guess}[example]{Theorem}

\newtheorem{coro}[example]{Corollary}
\newtheorem{llemma}[example]{Lemma}
\newtheorem{con}[example]{Conjecture}

\newtheorem{rem}[example]{Remark}

\def\qed{\hfill \rule{4pt}{7pt}}
\def\pf{\noindent {\it Proof. }}
\def\qed{\hfill \rule{4pt}{7pt}}

\begin{document}
\maketitle
\begin{abstract}
Let $G$ be a graph. We use $\chi(G)$ and $\omega(G)$ to denote the chromatic number and clique number of $G$ respectively. A $P_5$ is a path on 5 vertices, and an $HVN$ is a $K_4$ together with one more vertex which is adjacent to   exactly two vertices of $K_4$. Combining with some known result, in this paper we show that if $G$ is $(P_5, \textit{HVN})$-free, then $\chi(G)\leq \max\{\min\{16, \omega(G)+3\}, \omega(G)+1\}$. This upper bound is almost sharp.
\end{abstract}

\textit{Keywords}: $P_5$, HVN, chromatic number, clique number.
\section{Introduction}

Graphs appeared in this paper are all simple and connected.

Let $G$ be a graph with vertex set $V(G)$ and edge set $E(G)$. For any positive integar $k$, we use $P_k$ and $C_k$ to denote a \textit{path} and a \textit{cycle} on $k$ vertices respectively. Let $u\in V(G)$ and $S\subseteq V(G)$. We use $N_{G[S]}(u)$ to denote the set of neighbours of $u$ in $S$. If $u$ is adjacent to   all (no) vertices of $S$, then we say that {\it $u$ is complete (anticomplete) to $S$}. A subset $T\subseteq V(G)$ is said to be {\it complete (anticomplete) to $S$} if each vertex in $T$ is complete (anticomplete) to $S$. We say that {\it $T$ is adjacent to   $S$} if there exists $v\in T$ such that $N_{G[S]}(v)\neq\emptyset$. A subset $H$ of vertices is {\it homogeneous} if every vertex in $V(G)\backslash H$ is either complete or anticomplete to $H$.

A subset $S\subseteq V(G)$ is said to be a {\it clique} if the subgraph $G[S]$ induced by $S$ is a complete graph, and is said to be {\it stable} if $G[S]$ has no edges. The size of the largest clique of $G$ is said to be the {\it clique number} of $G$, denoted by $\omega(G)$. Let $k$ be a positive integer. A {\it$k$-coloring} of $G$ is a function $\ell: V(G)\mapsto \{1, \ldots, k\}$ such that $\ell(u)\neq \ell(v)$ if $u$ and $v$ are adjacent in $G$. The {\it chromatic number} of $G$, denoted by $\chi(G)$, is the minimum number $k$ such that $G$ admits a $k$-coloring.

A graph $G$ is said to be a {\it perfect graph} if $\chi(H)=\omega(H)$ for any induced subgraph $H$ of $G$, and so if $G$ is perfect, then $\chi(G)=\omega(G)$. It is easy to observe that in general, the chromatic number of a graph $G$ is no less than the clique number. However, it is much more complicated to determine a non-trivial upper bound for $\chi(G)$ with respect to $\omega(G)$. A family of graphs $\mathcal{G}$ is said to be {\it$\chi$-bounded} \cite{MR0382051} if there exists some function $f$ such that $\chi(G)\leq f(\omega(G))$ for every $G\in\mathcal{G}$, and $f$ is said to be the {\it binding function} of $G$. Erd\H{o}s \cite{MR102081} showed that for any positive integers $k$ and $\ell$, there exists a graph $G$ with $\chi(G)\geq k$ and no cycles of length less than $\ell$. This motivated the study of the chromatic number of $\mathcal{H}$-free graphs for some $\mathcal{H}$. A graph is {\it $H$-free} if $G$ does not have $H$ as an induced subgraph, and is $\mathcal{H}$-free if it is $H$-free for each $H\in \mathcal{H}$. Gy\'{a}rf\'{a}s \cite{MR0382051, MR951359}, and Sumner \cite{MR634555} proposed the following conjecture respectively.

\begin{con}\label{Gyarfas87}
For every tree $T$, $T$-free graphs are $\chi$-bounded.
\end{con}

Gy\'{a}rf\'{a}s \cite{MR951359} proved that if $G$ is $P_k$-free and $\omega(G)\ge 2$, then $\chi(G) \leq {(k-1)}^{\omega(G) - 1}$ for $k\ge 4$. Gravier {\it et al.} \cite{MR2009549} further improved the upper bound to $(k-2)^{\omega(G)-1}$. Since $P_4$-free graphs are perfect, determining an optimal binding function of $P_5$-free graphs attracts plenty of attention. Up to now, the best known upper bound for $P_5$-free graphs is given by Esperet {\it et al.} \cite{MR3010736}, who showed that if $G$ is $P_5$-free and $\omega(G)\ge 3$, then $\chi(G)\le 5\cdot3^{\omega(G)-3}$, and the bound is sharp for $\omega(G)=3$.  Choudum {\it et al.} \cite{MR2292667} conjectured that there exists a constant $c$ such that for every $P_5$-free graph $G$, $\chi(G)\le c\omega^2(G)$.

\begin{figure}[htbp]
\begin{center}
\includegraphics[scale=0.8]{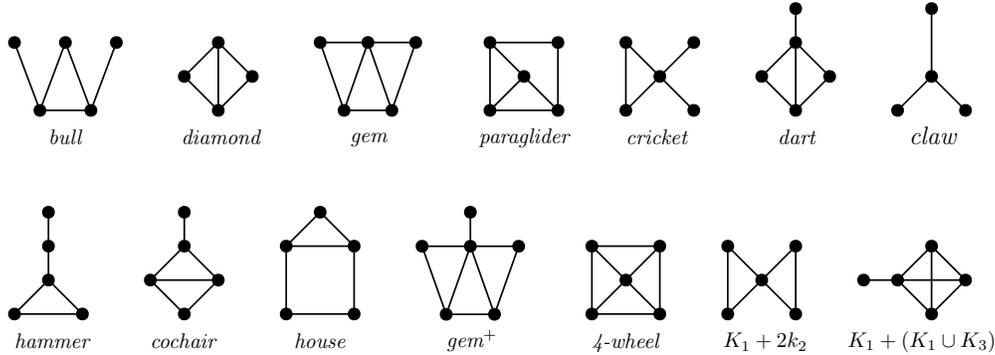}
\end{center}
\vskip -3pt
 {\small \caption{Illustration of some forbidden configurations}}
\label{fig-1}
\end{figure}

The study of $(P_5, H)$-free graphs has attracted plenty of attention, especially, when $|V(H)|\leq 5$. It is shown in \cite{MR634555} that all  $(P_5, K_3)$-free graphs are 3-colorable, and there exist many $(P_5, K_3)$-free graphs with chromatic number 3. Esperet {\it et al.} \cite{MR3010736} showed that every $(P_5, K_4)$-free graph $G$ satisfies $\chi(G)\leq 5$. If $G$ is a $(P_5, \textit{H})$-free graph where $H \in \{$hammer \cite{MR3906725}, bull \cite{MR3879962}, cochair and cricket \cite{DX}, house \cite{MR1360104}$\}$, then $G$ admits a ${\omega(G)+1\choose 2}$-coloring. Schiermeyer \cite{MR3488929} proved that $\chi(G)\le \omega^2(G)$ for ($P_5$, $H$)-free graphs $G$, where $H$ is a graph in $\{$claw, dart, diamond$\}$. If $G$ is $(P_5, \textit{gem})$-free, then it is shown in \cite{MR4174128} that $G$ is $\lceil\frac{5\omega(G)}{4}\rceil$-colorable. Dong {\it et al.} \cite{MR4434829} showed that $\chi(G)\le 2\omega^2(G)-\omega(G)-3$ if  $G$ is  $(P_5,  K_{2,3})$-free with $\omega(G)\ge 2$, $\chi(G)\le {\omega(G)+1\choose 2}$ if $G$ is $(P_5, C_5, K_{2, 3})$-free, and $\chi(G)\le {3\over 2}(\omega^2(G)-\omega(G))$ if  $G$ is   $(P_5, K_1+2K_2)$-free. In \cite{DXYX}, Dong {\it et al.} gave a characterization for $(P_5, K_1\cup K_3)$-free graphs, and  proved that $\chi(G)\le 2\omega(G)-1$ if $G$ is $(P_5, K_1\cup K_3)$-free. In the same paper, they further showed that  if $G$ is $(P_5, K_1+(K_1\cup K_3))$-free, then $\chi(G)\le $max$\{2\omega(G),15\}$, and they construct some $(P_5, K_1+( K_1\cup K_3))$-free graph with $\chi(G)=2\omega(G)$. Huang and Karthick \cite{MR4267041} showed that if $G$ is  $(P_5$, paraglider)-free, then $\chi(G)\le \lceil{3\omega(G)\over 2}\rceil$. Char and Karthick \cite{MR4366289} proved that a $(P_5, \textit{4-wheel})$-free graph $G$ is $\frac{3}{2}\omega(G)$-colorable. There are also interesting results for $H$ whose order is larger than $5$. Schiermeyer \cite{MR3488929} proved that $\chi(G)\le \omega^2(G)$  when $H$ is a gem${^+}$. A $(P_5, K_{2,t})$-free graph with $t\geq 2$ is $(\ell_t\cdot \omega(G)^t)$-colorable for some constant $\ell_t$ \cite{MR3898374}.

Recently, Trotignon and Pham \cite{MR3789676} posed a question asking for a polynomial binding function for $P_5$-free graphs. Notice that the existence of such a function would imply the Erd\H{o}s-Hajnal conjecture \cite{MR599767, MR1031262} for $P_5$-free graphs. This problem is still open even for $(P_5, H)$-free graphs for $H$ with small order. Chudnovsky and Sivaraman \cite{MR3879962} showed that $\chi(G)\le 2^{\omega(G)-1}$ if $G$ is ($P_5, C_5)$-free.  For $(P_5, P^c_5)$-free graphs, Fouquet {\it et al.} \cite{MR1360104} showed that there exists no linear binding function. We refer the readers to \cite{MR3898374, Scott2020} for more information about the topics related to $\chi$-boundedness.

\begin{figure}[htbp]
\begin{center}
\includegraphics[scale=0.8]{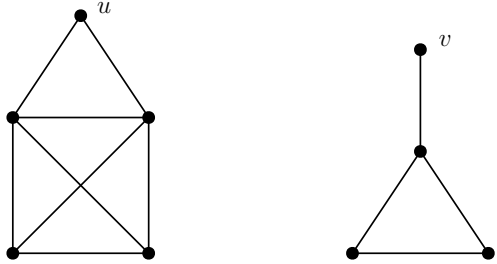}
\end{center}
\caption{(a): An HVN. (b): A paw.}
\label{fullpaw}
\end{figure}

An {\it HVN} is a graph consisting of a $K_4$ and an additional vertex $u$ where $u$ is adjacent to  exactly two vertices on $K_4$ (Figure~\ref{fullpaw}(a)). Karthick {\it et al.} \cite{MR3854117} proved that any $(2K_2, \textit{HVN})$-free graph $G$ satisfies $\chi(G)\leq \omega(G)+3$. Malyshev (see \cite[Lemma 16]{MR3449432}) proved that $\chi(G)\leq \max\{16, \omega(G)+1\}$ for any $(P_5, \textit{HVN})$-free graphs. Our main result generalizes the  result of \cite{MR3854117} to a larger family of graphs, and improves the result of \cite{MR3449432} when the graph in consideration  has clique number at most 13.

\begin{guess}\label{main}
Let $G$ be a $(P_5, \textit{HVN})$-free graph. Then $\chi(G)\leq \omega(G)+3$.
\end{guess}

Notice that $(P_5, K_4)$-free graphs are $5$-colorable and there are many $(P_5, K_4)$-free graphs with chromatic number 5 \cite{MR3010736}. Thus the bound in Theorem~\ref{main} is almost sharp for graphs with small clique number. Since $(P_5, \textit{HVN})$-free graphs can have induced cycle with five vertices, such graphs may not be perfect, and so the upper bound  $\omega(G)+1$ of \cite[Lemma 16]{MR3449432} is certainly tight for graphs with large clique number. Combining the conclusion of \cite[Lemma 16]{MR3449432}, we have that $\chi(G)\leq \max\{\min\{16, \omega(G)+3\}, \omega(G)+1\}$ for all $(P_5, \textit{HVN})$-free graphs.

Before presenting the sketch of the proof of Theorem~\ref{main}, we introduce two more structures. A {\it $T$-5-wheel} is a graph consisting of a chordless 5-cycle $C=v_{1}v_{2}v_{3}v_{4}v_{5}v_1$ and an additional vertex $x$  where $x$ is adjacent to   three consecutive vertices on $C$, say $v_1, v_2, v_5$ (Figure~\ref{5wheel}(a)). A {\it $Y$-5-wheel} is a graph consisting of a chordless 5-cycle $C=v_{1}v_{2}v_{3}v_{4}v_{5}v_1$ and an additional vertex $x$  where $x$ is adjacent to   three non-consecutive vertices on $C$, say $v_1, v_2, v_4$ (Figure~\ref{5wheel}(b)).

The proof of Theorem~\ref{main} is organized as following.  Section 2 contains the preliminaries which are crucial for the proof of Theorem~\ref{main}. In Section 3, we will first prove that $(P_5, \textit{HVN})$-free graphs with induced $T$-5-wheels are $(\omega+3)$-colorable (Theorem~\ref{t5}). Then in Section 4 we show that $(P_5, \textit{HVN}, T\textit{-5-wheel})$-free graphs with induced $Y$-5-wheels are $(\omega+3)$-colorable (Theorem~\ref{y5}). Finally, we finish the proof of Theorem~\ref{main} in the last section by showing that $(P_5, \textit{HVN}, T\textit{-5-wheel}, Y\textit{-5-wheel})$-free graphs are $(\omega+3)$-colorable.

\begin{figure}[htbp]
\begin{center}
\includegraphics[scale=0.8]{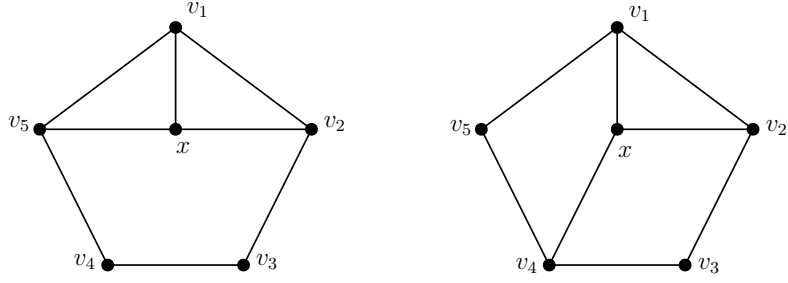}
\end{center}
\caption{(a): A $T$-5-wheel. (b): A $Y$-5-wheel.}
\label{5wheel}
\end{figure}

\section{Preliminaries}

The aim of this section is to present some results which will be cited repeatedly in Section 3, Section 4 and Section 5. A {\it paw} is a graph consisting of a $K_3$ and an additional vertex $v$ where $v$ has exactly one neighbour on $K_3$ (Figure~\ref{fullpaw}(b)). A {\it 5-ring} is a graph whose vertex set can be partitioned into 5 non-empty stable sets $R_1, \ldots, R_5$ such that $R_i$ is complete to $R_{i-1}\cup R_{i+1}$ and anticomplete to $R_{i-2}\cup R_{i+2}$ (with subscripts modulo 5) for each $1\leq i\leq 5$.

\begin{guess}\label{pawfree}\cite{MR947254}
A graph $G$ is paw-free if and only if $G$ is either complete multipartite or triangle-free.
\end{guess}

\begin{guess}\label{trianglefree}\cite{MR634555}
A graph $G$ is $(P_{5}, K_3)$-free if and only if $G$ is either bipartite or a 5-ring.
\end{guess}

\begin{llemma}\label{tfree}\cite[Corollary 2.2]{MR3010736}
Let $G$ be a $(P_5, K_3)$ graph. Let $S$ and $T$ ($T$ might be empty) be two disjoint stable sets where each vertex of $T$ has a neighbour in $S$. Then $G$ admits a $3$-coloring where $S$ receives $1$ and $T$ receives $2$.
\end{llemma}

As a direct consequence of Theorem~\ref{pawfree} and Lemma~\ref{tfree}, we have the following result.

\begin{coro}\label{colorpawfree}
Suppose that $G$ is a $(P_5, paw)$-free graph. Let $S$ and $T$ ($T$ might be empty) be two disjoint stable sets where each vertex of $T$ has a neighbour in $S$. Then $G$ admits a $(\omega(G)+1)$-coloring where $S$ and $T$ receive one color respectively.
\end{coro}

Let $G$ be a $P_{5}$-free graph and $C=v_{1}v_{2}v_{3}v_{4}v_{5}v_1$ be an induced 5-cycle of $G$.

\begin{llemma}\label{5cycle}\cite[Lemma 3.1]{MR3010736}
For each $i\in \{1, \ldots,5\}$, let
\begin{itemize}
\item[](1). $S=\{u | N_{C}(u)=\emptyset\}$;
\item[](2). $R_{i}=\{u | N_{C}(u)=\{v_{i-1}, v_{i+1}\}\}$;
\item[](3). $\bar{R}_{i}=\{u | N_{C}(u)=\{v_{i-1}, v_{i}, v_{i+1}\}\}$;
\item[](4). $Y_{i}=\{u | N_{C}(u)=\{v_{i-2}, v_{i}, v_{i+2}\}\}$;
\item[](5). $P^{i}=\{u | N_{C}(u)=\{v_{i-1}, v_{i}, v_{i+1}, v_{i+2}\}\}$;
\item[](6). $T=\{u | N_{C}(u)=V(C)\}$.
\end{itemize}
Then $V(G)=T\bigcup S\bigcup\cup_{1\leq i\leq 5}(R_{i}\cup \bar{R}_{i}\cup Y_i\cup P^i)$.
\end{llemma}


\begin{llemma}\label{5cyclest}
Suppose that $G$ is a $(P_{5}, \textit{HVN})$-free graph. Then for $1\leq i\leq 5$
\begin{itemize}
\item[](a). $T\bigcup(\cup_{1\leq i\leq 5}P^{i})$ is stable, $Y_{i}$ is stable,
\item[](b). $T$ is anticomplete to $\bar{R}_{i}$ and $Y_{i}$,
\item[](c). $Y_{i}$ is anticomplete to $P^{i+1}, P^{i+2}, P^{i-2}, \bar{R}_{i+2}\cup \bar{R}_{i-2}$,
\item[](d). $P^{i}$ is anticomplete to $\bar{R}_{j}$ with $j\neq i-2$,
\item[](e). $R_{i}\cup \bar{R}_{i}$ is complete to $R_{i+1}\cup \bar{R}_{i+1}$,
\item[](f). either $\bar{R}_{i}=\emptyset$ or $\bar{R}_{i+1}=\emptyset$,
\item[](g). $S$ is anticomplete to $R_{i}\cup \bar{R}_{i}$.
\end{itemize}
\end{llemma}

\begin{proof}
Let $u,v\in T\bigcup(\cup_{1\leq i\leq 5}P^{i})$. If $u, v\in T$, then $uv\notin E(G)$, otherwise $\{u, v, v_{1}, v_{2}, v_{4}\}$ induces an HVN. If $u\in T$ and $v\in P^{i}$ for some $i$, then $uv\notin E(G)$, otherwise $\{u, v, v_{i-1}, v_{i}, v_{i+2}\}$ induces an HVN. Suppose that $u,v\in \cup_{1\leq i\leq 5}P^{i}$. If $uv\in E(G)$, then for some $1\leq i\leq5$, either $\{u,v, v_{i+1}, v_{i+2}, v_{i-1}\}$ or $\{u,v, v_{i+1}, v_{i+2},v_{i}\}$ induces an HVN, which leads to a contradiction. Now let $u,v\in Y_{i}$ such that $uv\in E(G)$. Then $\{u, v, v_{i-2}, v_{i+2}, v_{i}\}$ induces an HVN. Thus this proves (a).

Let $u\in T$ and $v\in \bigcup_{1\leq i\leq 5}(\bar{R}_{i}\cup Y_{i}\cup P^{i})$. Suppose that $uv\in E(G)$.  If $v\in \bar{R}_{i}$, then $\{u, v, v_{i}, v_{i+1},v_{i+2}\}$ induces an HVN. If $v\in Y_{i}$, then $\{u, v, v_{i-2}, v_{i+2}, v_{i-1}\}$ induces an HVN. Hence this completes the proof of (b).

For $1\leq i\leq 5$, let $u\in Y_{i}$ and $v\in P^{i+1}\cup P^{i+2}$. If $uv\in E(G)$, then $\{u,v,v_{i-2}, v_{i+2},v_{i+1}\}$ induces an HVN, which leads to a contradiction. Suppose that $v\in P^{i-2}$, then $uv\notin E(G)$ as otherwise $\{u,v,v_{i-2}, v_{i+2},v_{i-1}\}$ induces an HVN. Now let $w\in \bar{R}_{i+2}\cup \bar{R}_{i-2}$. Then $uw\notin E(G)$, otherwise either $\{v_{i-2}, v_{i+2}, u, w, v_{i+1}\}$ or $\{v_{i-2}, v_{i+2}, u, w, v_{i-1}\}$ induces an HVN, which proves (c).

For $1\leq i\leq 5$, let $u\in P^{i}$ and $v\in \bar{R}_{j}$ with $j\neq i-2$. Suppose that $j=i$. Then $uv\notin E(G)$, otherwise $\{u, v, v_{i}, v_{i+1}, v_{i+2}\}$ induces an HVN, which leads to a contradiction. Similarly, we have that if $j=i-1,i+1, i+2$, then $uv\notin E(G)$. Thus this proves (d).

For $1\leq i\leq 5$, let $u\in R_{i}\cup \bar{R}_{i}$ and $v\in R_{i+1}\cup \bar{R}_{i+1}$. If $uv\notin E(G)$, then $vv_{i+2}v_{i-2}v_{i-1}u=P_{5}$. Hence this proves (e).

Suppose that $\bar{R}_{i}$ and $\bar{R}_{i+1}$ are both non-empty. Then by (e), we have that $\bar{R}_{i}$ is complete to $\bar{R}_{i+1}$. Let $u\in \bar{R}_{i}$ and $v\in \bar{R}_{i+1}$. Then $\{u,v, v_{i}, v_{i+1}, v_{i+2}\}$ induces an HVN, which leads to a contradiction. Hence this proves (f).

Let $s\in S$ and $u\in R_{i}\cup \bar{R}_{i}$. If $us\in E(G)$, then $suv_{i+1}v_{i+2}v_{i-2}=P_{5}$, and so this completes the proof of (g).
\end{proof}

\begin{llemma}\label{homogeneous}
Let $G$ be a $(P_{5}, \textit{HVN})$-free graph. Suppose that $H$ is an induced subgraph of $G$ and $H$ contains an induced paw. Let $S=\{u\in V(G)| N_{H}(u)=\emptyset\}$.  Then $S$ is homogeneous in $G$ and $G[S]$ is $\omega(G)$-colorable.
\end{llemma}

\begin{proof}
We may assume that $G[S]$ is connected. Suppose to the contrary that $S$ is not homogeneous in $G$, that is, there exist $w\in V(G)\backslash S$ and $s_{1}, s_{2}\in S$ such that $s_{1}s_{2}\in E(G)$, $s_{1}w\in E(G)$ and $s_{2}w\notin E(G)$. Since $w\in V(G)\backslash S$, there exists $u\in V(H)$ such that $wu\in E(G)$. Notice that $w$ is not complete to $V(H)$ as $G$ is HVN-free and $H$ contains an induced paw. Thus there exists $v\in V(H)$ such that $uv\in E(G)$ and $vw\notin E(G)$. Hence $vuws_{1}s_2=P_5$, which leads to a contradiction. Thus $S$ is homogeneous in $G$. Since $G$ is connected, there must exist a vertex in $V(G)\backslash S$ which is complete to $S$, and so  $G[S]$ is paw-free. Therefore by Lemma~\ref{colorpawfree} we have that $G[S]$ is $\omega(G)$-colorable.
\end{proof}

\section{$T$-5-Wheel}

Recall that a $T$-5-wheel is a graph consisting of a chordless 5-cycle $C=v_{1}v_{2}v_{3}v_{4}v_{5}v_1$ and an additional vertex $x$  where $x$ is adjacent to   three consecutive vertices on $C$, say $v_1, v_2, v_5$. Let $G$ be a $(P_{5}, \textit{HVN})$-free graph and $\Sigma$ be an induced $T$-5-wheel of $G$. In this section, we will show that $G$ is $(\omega(G)+3)$-colorable.

For $1\leq i\leq 5$, let $R_i, \bar{R}_i, Y_i, P^i$ and $T$ be as defined in Lemma~\ref{5cycle}, and in this section, we especially use these notations to denote the subsets of vertices adjacent to   $C$ but not adjacent to   $x$. Clearly, $x\in \bar{R}_1$, and so $\bar{R}_1\neq\emptyset$.  Now define $S=\{u | N_{\Sigma}(u)=\emptyset\}$,
\begin{itemize}
\item $R_{i,x}=\{u | N_{\Sigma}(u)=\{v_{i-1}, v_{i+1}, x\}\}$,
\item $\bar{R}_{i,x}=\{u | N_{\Sigma}(u)=\{v_{i-1}, v_{i}, v_{i+1},x\}\}$,
\item $Y_{i,x}=\{u | N_{\Sigma}(u)=\{v_{i-2}, v_{i}, v_{i+2}, x\}\}$,
\item $P^{i,x}=\{u | N_{\Sigma}(u)=\{v_{i-1}, v_{i}, v_{i+1}, v_{i+2}, x\}\}$, and
\item $T_x=\{u | N_{\Sigma}(u)=V(\Sigma)\}$.
\end{itemize}

\begin{rem}\label{rem1}
Notice that $x$ is an additional vertex to $C$, and so all results of Lemma~\ref{5cyclest}  hold for $R_{i}, \bar{R}_{i}, Y_{i}, P^{i}, T$ defined in Lemma~\ref{5cycle} and $R_{i,x}, \bar{R}_{i,x}, Y_{i,x}, P^{i,x}, S$ defined above.
\end{rem}
Since $x\in \bar{R}_1$, it follows from Lemma~\ref{5cyclest}(b) that $T_x=\emptyset$. By Lemma~\ref{5cyclest}(f) and Remark~\ref{rem1}, we may always assume that $\bar{R}_4\cup \bar{R}_{4, x}=\emptyset$.

\begin{llemma}\label{T5}
Let $G$ be a $(P_{5}, \textit{HVN})$-free graph. Suppose that $G$ contains an induced $T$-5-wheel $\Sigma$. Then
$$
V(G)=N_{G}(x)\bigcup R_{1}\bigcup R_3\bigcup R_4\bigcup \bar{R}_1\bigcup \bar{R}_3\bigcup Y_{2}\bigcup Y_{5}\bigcup (\cup_{i=2,3,4}P^{i})\bigcup T\bigcup S,
$$ and
$$
N_{G}(x)=R_{1,x}\bigcup R_{2,x}\bigcup R_{5,x} \bigcup\bar{R}_{1,x}\bigcup \bar{R}_{3,x}\bigcup Y_{1,x}\bigcup  P^{3,x}.
$$
\end{llemma}

\pf Recall that $C=v_{1}v_{2}v_{3}v_{4}v_{5}v_1$ is the induced 5-cycle in $\Sigma$. Clearly, $v_i\in R_{i,x}$ for $i=1,2,5$,  and $v_j\in R_j$ for $j=3,4$. Recall that $x\in \bar{R}_1$. Now let $u\in V(G)\backslash V(\Sigma)$.

{\it \bf Case 1: $u$ is not adjacent to   $C$.} If $u$ is adjacent to  $x$, then $uxv_{2}v_{3}v_{4}=P_{5}$, which leads to a contradiction. Thus in this case $u\in S$.

{\it\bf Case 2: $N_C(u)=\{v_{i-1}, v_{i+1}\}$.} Suppose that $i=1$. Then either $u\in R_{1}$ or $u\in R_{1,x}$.  Suppose that $i=3$. If $ux\in E(G)$, then $v_{3}v_{4}uxv_{1}=P_{5}$, and so $u\in R_{3}$. Similarly, if $i=4$, then $u\in R_{4}$. Suppose that $i=2$. Then $u$ must be adjacent to  $x$, otherwise $xv_{1}uv_{3}v_{4}=P_{5}$. Thus $u\in R_{2,x}$. Similarly, if $i=5$, then $u\in R_{5,x}$.

{\it\bf Case 3: $N_C(u)=\{v_{i-1},v_i, v_{i+1}\}$.} Since $\bar{R}_1\neq\emptyset$, by Lemma~\ref{5cyclest}(f), we have that $\bar{R}_2\cup \bar{R}_{2,x}=\bar{R}_5\cup \bar{R}_{5,x}=\emptyset$. Hence in this case,  $u\in \bar{R}_{1}\cup \bar{R}_{1,x}$ or $u\in \bar{R}_{3}\cup \bar{R}_{3,x}$.

{\it\bf Case 4: $N_C(u)=\{v_{i-2}, v_i, v_{i+2}\}$.} Suppose that $i=1$. Since $uv_{4}v_{5}xv_{2}$ cannot be an induced $P_{5}$, we have that $ux\in E(G)$, and so $u\in Y_{1,x}$. Suppose that $i=2$. If $ux\in E(G)$, then $v_{1}xuv_{4}v_{3}=P_{5}$, which leads to a contradiction. Thus $ux\notin E(G)$ and $u\in Y_{2}$. Similarly, if $i=5$, then $ux\notin E(G)$ and so $u\in Y_{5}$. Suppose that $i=3$. Then $ux\in E(G)$, otherwise $xv_{1}uv_{3}v_{4}=P_{5}$. This implies that $\{v_{5}, x, v_{1}, u, v_{2}\}$ induces an HVN. Thus $i\neq 3$, and similarly, $i\neq 4$.

{\it\bf Case 5: $N_C(u)=\{v_{i-1},v_i, v_{i+1}, v_{i+2}\}$.} Suppose that $i=1$. If $ux\notin E(G)$, then $xv_{1}uv_{3}v_{4}=P_{5}$. If $ux\in E(G)$, then $\{u,v_{1},x,v_{2},v_{3}\}$ induces an HVN, and so we have that $i\neq 1$. By symmetry, we have that $i\neq 5$. Suppose that $i=2$. If $ux\in E(G)$, then $\{u,v_{1},x,v_{2},v_{3}\}$ induces an HVN, and so $u\in P^{2}$. Similarly, if $i=4$, then $ux\notin E(G)$ and $u\in P^{4}$. If $i=3$, then either $u\in P^{3}$ or $P^{3,x}$.

{\it\bf Case 6: $uv_i\in E(G)$ for all $1\leq i\leq 5$.} Since $T_x=\emptyset$,  we have that  $u\in T$. \qed

\begin{llemma}\label{T5stable}
Let $G$ be a $(P_{5}, \textit{HVN})$-free graph. Suppose that $G$ contains an induced $T$-5-wheel $\Sigma$. Then
\begin{itemize}
\item[](1). $R_{4}$ ($R_{3}\cup \bar{R}_{3}\cup \bar{R}_{3,x}$) is a stable set if $R_{3}$ ($R_{4}$) is not stable,
\item[](2). $R_{2,x}\cup R_{5,x}$ is stable,
\item[](3). $Y_{2}, Y_{5}, Y_{1,x}, P^{i} ($i=2,3,4$), P^{3,x}, T$ are stable sets, and
\item[](4). $Y_{1, x}$ is anticomplete to $\bar{R}_{3}\cup \bar{R}_{3,x}$.
\end{itemize}
\end{llemma}

\begin{proof}
Clearly, (3) directly follows from Lemma~\ref{5cyclest}(a) and Remark~\ref{rem1}, and (4) follows from Lemma~\ref{5cyclest}(c) and Remark~\ref{rem1}. Thus it is left to prove (1) and (2).

Now suppose that $R_{3}$ and $R_{4}$ are both non-empty. It follows from Lemma~\ref{5cyclest}(e) that $R_{3}$ is complete to $R_{4}$. Suppose that $R_{3}$ is not stable, and let $u, v\in R_{3}$ such that $uv\in E(G)$. Suppose that $R_{4}$ is not stable, and let $y,z\in R_{4}$ such that $yz\in E(G)$. Thus $\{v_{5}, u,v,y,z\}$ induces an HVN. Thus $R_{4}$ must be stable. By similar arguments, if $R_4$ is not stable, then $R_{3}\cup \bar{R}_{3}\cup \bar{R}_{3,x}$ must be stable. Thus this proves (1).


Let $u, v\in R_{2,x}\cup R_{5,x}$ such that $uv\in E(G)$. First suppose that $u, v\in R_{2,x}$. Then $\{u, v, v_{1}, x,v_{5}\}$ induces an HVN, which leads to a contradiction. Thus $R_{2,x}$ is stable. Similarly, we can show that $R_{5,x}$ is stable. Now suppose that $u\in R_{2,x}$ and $v\in R_{5,x}$. Notice that if $u=v_2$ and $v=v_5$, then $uv\notin E(G)$. Thus we may assume that $u\neq v_2$. Then $\{u, v, v_{1}, x, v_{2}\}$ induces an HVN. Thus $R_{2,x}$ must be anticomplete to $R_{5,x}$, which completes the proof of (2).
\end{proof}

\begin{llemma}\label{T5S}
Let $G$ be a $(P_{5}, \textit{HVN})$-free graph. Suppose that $G$ contains an induced $T$-5-wheel $\Sigma$. Then $S$ is
\begin{itemize}
\item[](i). anticomplete to $\cup_{i=1,3,4}R_{i}\bigcup (\cup_{i=1,2,5}R_{i,x})\bigcup [\cup_{i=1,3}(\bar{R}_i\cup \bar{R}_{i,x})]$, and
\item[](ii). anticomplete to $P^{2}\cup P^{4}$.
\end{itemize}
\end{llemma}

\begin{proof}
It follows from Lemma~\ref{5cyclest}(g) and Remark~\ref{rem1} that {\it(i)} holds. Now let $s\in S, u\in P^{2}\cup P^{4}$ such that $su\in E(G)$.  If $u\in P^2$, then $suv_{4}v_{5}x=P_{5}$. If $u\in P^4$, then $suv_{3}v_{2}x=P_{5}$, which leads to a contradiction. Thus $S$ is anticomplete to $P^{2}\cup P^{4}$.
\end{proof}

\begin{coro}\label{SN}
$N_{V(G)\backslash S}(S)\subseteq Y_{2}\cup Y_{5}\cup Y_{1,x}\cup P^{3}\cup P^{3, x}\cup T$.
\end{coro}

\begin{guess}\label{t5}
Let $G$ be a $(P_{5}, \textit{HVN})$-free graph. Suppose that $G$ contains an induced $T$-5-wheel $\Sigma$. Then $\chi(G)\leq \omega(G)+3$.
\end{guess}

\begin{proof}
It follows from Lemma~\ref{T5stable} that $R_{3}$ and $R_{4}$ cannot be both non-stable. Thus we are going to prove this theorem by making arguments on the following two cases.

{\it\bf Case 1:} Firstly, suppose that $R_{4}$ is not stable. Then it follows from Lemma~\ref{T5stable}(1) that $R_{3}\cup \bar{R}_{3}\cup \bar{R}_{3,x}$ is stable. Now let
\begin{align*}
A = & R_{1}\cup R_{1,x}\cup\bar{R}_{1}\cup\bar{R}_{1,x}\cup R_{4}\cup Y_{2}\cup Y_{5}\cup P^{3}\cup P^{3,x}\cup P^{4}\cup T,\\
B = & R_{3},\\
M = & \bar{R}_{3}\cup\bar{R}_{3,x}\cup Y_{1, x}\cup P^{2},\\
D = & R_{2,x}\cup R_{5,x}.
\end{align*}
Clearly we have that $V(G)\backslash S=A\cup B\cup M\cup D$ and $B$ is stable. By Remark~\ref{rem1}, Lemma~\ref{5cyclest}(c,d) and Lemma~\ref{T5stable}(2,3,4), we have that $M$ and $D$ are stable sets. Since all the vertices of $A$ is adjacent to   $v_{5}$ and $G$ is HVN-free, we have that $G[A]$ is paw-free. By Lemma~\ref{colorpawfree}, $\chi(G[A])\leq \omega(G[A])+1\leq (\omega(G)-1)+1=\omega(G)$. Thus $G[A]$ is $\omega(G)$-colorable, and let the colors be $\{c_{1}, c_2, \ldots, c_{\omega}\}$.

Let $Y_{5, 1}$ be the set of vertices in $Y_5$ which have neighbours in $Y_{2}\cup P^{4}\cup P^{3}\cup P^{3,x}\cup T$, and $Y_{5, 2}=Y_{5}\backslash Y_{5, 1}$. Notice that by Remark~\ref{rem1} and Lemma~\ref{5cyclest}(a, b, c), we have that $Y_{2}\cup P^{4}\cup P^{3}\cup P^{3,x}\cup T$ is stable, which implies that $(Y_{2}\cup P^{4}\cup P^{3}\cup P^{3,x}\cup T)\cup Y_{5,2}$ is stable as well.

By Corollary~\ref{colorpawfree},  we can always color $G[A]$ such that $(Y_{2}\cup P^{4}\cup P^{3}\cup P^{3,x}\cup T)\cup Y_{5,2}$ is colored by $c_1$, and $Y_{5, 1}$ is colored by $c_{2}$.  Let $c_{\omega+1}, c_{\omega+2}$ and $c_{\omega+3}$ be three new colors. Since $\Sigma$ contains an induced paw, by Lemma~\ref{homogeneous},  we have that $S$ is $\omega(G)$-colorable. It follows from Corollary~\ref{SN} that we can use $c_3, \ldots, c_{\omega}, c_{\omega+1}, c_{\omega+2}$ on $S$, use $c_{\omega+1}, c_{\omega+2}$ on $B,D$ respectively, and color $M$ by $c_{\omega+3}$. Hence $G$ is $(\omega(G)+3)$-colorable.

Now suppose that $R_3$ is not stable. Then $R_4$ is stable. Define
\begin{align*}
A = & R_{1}\cup R_{1,x}\cup\bar{R}_{1}\cup\bar{R}_{1, x}\cup R_{3}\cup \bar{R}_3\cup\bar{R}_{3,x}\cup Y_{2}\cup Y_{5}\cup P^{2}\cup P^{3}\cup P^{3,x}\cup T,\\
B = & R_{4},\\
M = & P^{4}\cup Y_{1, x},\\
D = & R_{2,x}\cup R_{5,x}.
\end{align*}
Observe that $G[A]$ is still $\omega(G)$-colorable as each vertex of $A$ is adjacent to   $v_2$. By Lemma~\ref{5cyclest}(c) and Lemma~\ref{T5stable}(2,3), $M$ and $D$ are stable sets. Then with the similar arguments we can show that $G$ is $(\omega(G)+3)$-colorable.

{\it\bf Case 2:} Suppose that $R_{3}, R_{4}$ are both stable (or both empty). Let
\begin{align*}
A = & R_{1}\cup R_{1,x}\cup\bar{R}_{1}\cup\bar{R}_{1,x}\cup R_{3}\cup\bar{R}_3\cup\bar{R}_{3,x}\cup Y_{2}\cup Y_{5}\cup P^{2}\cup P^{3}\cup P^{3,x}\cup T,\\
B = & R_{4},\\
M = & P^{4}\cup Y_{1, x},\\
D = & R_{2,x}\cup R_{5,x}.
\end{align*}
By Remark~\ref{rem1}, Lemma~\ref{5cyclest}(c) and Lemma~\ref{T5stable}(2,3), we have that $B$, $M$ and $D$ are all stable sets. Since all the vertices of $A$ are adjacent to   $v_{2}$ and $G$ is HVN-free, we have that $G[A]$ is paw-free, and so $G[A]$ is $\omega(G)$-colorable. Let $c_{1}, c_2, \ldots, c_{\omega}$ be the colors used on $G[A]$.

Let $Y_{2, 1}$ be the set of vertices of $Y_2$ which have neighbours in $Y_{5}\cup P^{2}\cup P^{3}\cup P^{3,x}\cup T$, and $Y_{2, 2}=Y_{2}\backslash Y_{2, 1}$. Notice that by Remark~\ref{rem1} and Lemma~\ref{5cyclest}(a, b, c), we have that $Y_{5}\cup P^{2}\cup P^{3}\cup P^{3,x}\cup T$ is stable, which implies that $(Y_{5}\cup P^{2}\cup P^{3}\cup P^{3,x}\cup T)\cup Y_{2,2}$ is stable as well.

By the same arguments of Case 1, we may conclude that $Y_{5}\cup P^{2}\cup P^{3}\cup P^{3,x}\cup T\cup Y_{2}$ can be colored by at most two colors, say $c_1$ and $c_{2}$.

Recall that $S$ is $\omega(G)$-colorable as $\Sigma$ contains a paw. It follows from Corollary~\ref{SN} that we can use $c_3, \ldots, c_{\omega}$ together with two new colors $c_{\omega+1}, c_{\omega+2}$ on $S$, and use $c_{\omega+1}, c_{\omega+2}$ on $B,D$, respectively. Hence $G$ is $(\omega(G)+3)$-colorable by coloring $M$ by an additional color $c_{\omega+3}$.
\end{proof}

\section{Y-5-Wheel}

Recall that a $Y$-5-wheel is a graph consisting of a chordless 5-cycle $C=v_{1}v_{2}v_{3}v_{4}v_{5}v_1$ and an additional vertex $x$  where $x$ is adjacent to   three non-consecutive vertices on $C$, say $v_1, v_2, v_4$. Let $G$ be a $(P_{5}, \textit{HVN}, T\textit{-5-wheel})$-free graph and $\Sigma$ be an induced $Y$-5-wheel of $G$. In this section, we will show that $G$ is $(\omega(G)+3)$-colorable.

For $1\leq i\leq 5$, let $R_i, \bar{R}_i, Y_i, P^i$ and $T$ be as defined in Lemma~\ref{5cycle}, and in this section, we still use these notations to denote the vertices adjacent to   $C$ but not adjacent to   $x$. Obviously $x\in Y_4$.  Now define $S=\{u | N_{\Sigma}(u)=\emptyset\}$,
\begin{itemize}
\item $R_{i,x}=\{u | N_{\Sigma}(u)=\{v_{i-1}, v_{i+1}, x\}\}$,
\item $\bar{R}_{i,x}=\{u | N_{\Sigma}(u)=\{v_{i-1}, v_{i}, v_{i+1},x\}\}$,
\item $Y_{i,x}=\{u | N_{\Sigma}(u)=\{v_{i-2}, v_{i}, v_{i+2}, x\}\}$,
\item $P^{i,x}=\{u | N_{\Sigma}(u)=\{v_{i-1}, v_{i}, v_{i+1}, v_{i+2}, x\}\}$, and
\item $T_x=\{u | N_{\Sigma}(u)=V(\Sigma)\}$.
\end{itemize}

\begin{rem}\label{rem2}
Since $x$ is an additional vertex to $C$, all results of Lemma~\ref{5cyclest} hold for $R_{i}, \bar{R}_{i}, Y_{i}, P^{i}, T$ defined in Lemma~\ref{5cycle} and $R_{i,x}, \bar{R}_{i,x}, Y_{i,x}, P^{i,x}, S$ defined in this section.
\end{rem}

\begin{llemma}\label{Y5}
Let $G$ be a $(P_{5}, \textit{HVN})$-free graph. Suppose that $G$ contains an induced $Y$-5-wheel $\Sigma$ and contains no induced $T$-5-wheel. Then
\begin{equation*}
V(G)= X\bigcup\left[\cup_{i\neq 4}(R_{i}\cup Y_{i,x})\right]\bigcup \left[\cup_{1\leq i\leq 5}(P^{i}\cup R_{i,x})\right]\bigcup (\cup_{i=1,2,4}Y_{i})\bigcup P^{3,x}\bigcup P^{4,x}\bigcup T\bigcup S
\end{equation*}
where $X=\{u | N_{\Sigma}(u)=\{x\}\}$.
\end{llemma}

Notice that Lemma~\ref{Y5} implies that $N_{G}(x)=X\bigcup (\cup_{1\leq i\leq 5}R_{i,x})\bigcup (\cup_{i\neq 4}Y_{i,x})\bigcup P^{3,x}\bigcup P^{4,x}$.

\begin{proof}
Let $u\in V(G)$. First suppose that $u=x\in V(\Sigma)$, then $u\in Y_4$. If $u=v_i$ with $i=3, 5$, then $u\in R_3\cup R_5$. If $i=1, 2$ or $i=4$, then $u\in R_{i,x}$.

Now suppose that $u\in V(G)\backslash V(\Sigma)$. Since $G$ contains no induced $T$-5-wheel, we have that $N_{C}(u)\neq\{v_{i-1}, v_i,v_{i+1}\}$.

{\it \bf Case 1: $u$ is not adjacent to   $C$.} If $u$ is adjacent to  $x$, then $u\in X$. Otherwise, $u\in S$.

{\it\bf Case 2: $N_{C}(u)=\{v_{i-1}, v_{i+1}\}$.} If $i=4$, then $u\in R_{4,x}$, otherwise $v_5uv_3v_2x=P_5$. For $i\neq 4$, we may observe that $u\in R_i\cup R_{i,x}$.

{\it\bf Case 3: $N_{C}(u)=\{v_{i-2}, v_i, v_{i+2}\}$.} When $i=1, 2$, then $u\in Y_i$ if $ux\notin E(G)$, and $u\in Y_{i, x}$ if $ux\in E(G)$. Suppose that $i=3$. If $ux\notin E(G)$,  then $v_{5}uv_3v_2x=P_{5}$. This implies a contradiction, and so $u\in Y_{3,x}$. Similarly, we may show that $u\in Y_{5,x}$ when $i=5$. Now suppose that $i=4$. Notice that if $ux\in E(G)$, then $\{v_1,u,v_2,x,v_4\}$ induces HVN, which leads to a contradiction. Thus $u\in Y_4$ when $i=4$.

{\it\bf Case 4: $N_{C}(u)=\{v_{i-1}, v_i, v_{i+1},v_{i+2}\}$.} Since $x\in Y_4$, It is easy to observe that by Lemma~\ref{5cyclest}(c), $u\in P^{i}\cup P^{i,x}$ for $i=3$ or $4$, and $u\in P^i$ for $i=1,2,5$.

{\it\bf Case 5: $N_{C}(u)=V(C)$.} Since $x\in Y_4$, $T_x=\emptyset$ by Lemma~\ref{5cyclest}(b), and so $u\in T$.
\end{proof}

\begin{llemma}\label{Y5stable}
Let $G$ be a $(P_{5}, \textit{HVN})$-free graph. Suppose that $G$ contains an induced $Y$-5-wheel $\Sigma$ and contains no induced $T$-5-wheel. Then
\begin{itemize}
\item[](1). $R_{i}$ ($i\neq 4$) and $P^i$ ($1\leq i\leq 5$) are stable, and $Y_{i}$ is stable ($i=1,2,4$).
\item[](2a). If $R_{1}$ is not empty, then $R_{2}$  is empty.
\item[](2b). $R_1$ is anticomplete to $R_3$.
\item[](2c). $R_2$ is anticomplete to $R_5$.
\item[](3). $Y_{2}$ is anticomplete to $R_1\cup R_3$, $Y_1$ is anticomplete to $R_2\cup R_5$.
\item[](4). $P^3$ is anticomplete to $R_1\cup R_3$, $P^4$ is anticomplete to $R_2\cup R_5$.
\end{itemize}
\end{llemma}

\begin{proof}
{\it (1).} Let $u, v\in R_i$ for some $i\neq 4$ such that $uv\in E(G)$. Then $\{u, v, v_{i-1}, v_{i-2}, v_{i+2}, v_{i+1}\}$ induces a $T$-5-wheel, which leads to a contradiction as $G$ contains no induced $T$-5-wheel. Thus $R_{i}$ is stable for all $i\neq 4$. It follows from Lemma~\ref{5cyclest}(a) that $P^i$ ($1\leq i\leq 5$) is stable and $Y_{i}$ is stable ($i=1,2,4$). Thus this proves (1).

{\it (2a).} Let $u\in R_1$. Suppose to the contrary that $R_{2}$ is not empty, and let $v\in R_{2}$. By Lemma~\ref{5cyclest}(e), we have that $uv\in E(G)$. Then $vuv_2xv_4=P_5$, which leads to a contradiction. Thus $R_2=\emptyset$.

{\it (2b).} Let $u\in R_1$ and $v\in R_3$. Suppose that $uv\in E(G)$. Then $xv_1v_5uv=P_5$, which implies a contradiction.

{\it (2c).} Let $u\in R_{2}$ and $v\in R_5$, and suppose that $uv\in E(G)$. Then $uvv_4xv_2=P_5$, which leads to a contradiction. Thus $R_2$ is anticomplete to $R_5$.

{\it (3).} Let $u\in Y_2$ and $v\in R_i$ for $i=1,3$. First suppose that $i=1$. If $uv\in E(G)$, then $vuv_4xv_1=P_5$, which leads to a contradiction. Now suppose that $i=3$. If $uv\in E(G)$, then $vuv_5v_1x=P_5$, a contradiction.

Let $u\in Y_1$ and $v\in R_2\cup R_5$. Suppose that $v\in R_2$. Then $uv\notin E(G)$, otherwise $vuv_4xv_2=P_5$. Now suppose that $v\in R_5$. Then $uv\notin E(G)$, otherwise $vuv_3v_2x=P_5$. Hence this completes the proof of (3).

{\it (4).} Let $u\in P^3$ and $v\in R_1\cup R_3$. Suppose that $v\in R_1$. Then  $uv\notin E(G)$, otherwise $vuv_4xv_1=P_5$. Suppose that $v\in R_3$. If $uv\in E(G)$, then $vuv_5v_1x=P_5$, which leads to a contradiction. Thus $P^3$ is anticomplete to $R_1\cup R_3$.

Finally, let $u\in P^4$ and $v\in R_2\cup R_5$. Suppose that $v\in R_2$. If $uv\in E(G)$, then $vuv_4xv_2=P_5$, which leads to a contradiction. Thus we may assume that $v\in R_5$. Then $uv\notin E(G)$, otherwise $vuv_3v_2x=P_5$. Thus $P^4$ is anticomplete to $R_2\cup R_5$.
\end{proof}

\begin{llemma}\label{Y5S}
Let $G$ be a $(P_{5}, \textit{HVN})$-free graph. Suppose that $G$ contains an induced $Y$-5-wheel $\Sigma$ and contains no induced $T$-5-wheel. Then
\begin{itemize}
\item[](i). $N_{V(G)\backslash S}(S)\subseteq Y_4\bigcup (\cup_{i\neq 4} Y_{i,x})\bigcup P^2\bigcup P^5\bigcup P^{3,x}\bigcup P^{4,x}\bigcup T$.
\item[](ii). Suppose that $|S|\geq 2$ and $G[S]$ is connected. Then
\begin{itemize}
\item[]a. each vertex of $Y_4\bigcup (\cup_{i\neq 4}Y_{i,x})$ is either complete or anticomplete to $S$, and
\item[]b. $S$ is adjacent to  at most one of $Y_4$, $Y_{1, x}, Y_{2,x}, Y_{3,x}$ and $Y_{5,x}$.
\end{itemize}
\end{itemize}
\end{llemma}

\begin{proof}
Recall that by Lemma~\ref{Y5}, we have that
\begin{equation*}
V(G)\backslash S= X\bigcup\left[\cup_{i\neq 4}(R_{i}\cup Y_{i,x})\right]\bigcup \left[\cup_{1\leq i\leq 5}(P^{i}\cup R_{i,x})\right]\bigcup (\cup_{i=1,2,4}Y_{i})\bigcup P^{3,x}\bigcup P^{4,x}\bigcup T.
\end{equation*}

{\it (i):} Let $s\in S$ and $u\in N_{V(G)\backslash S}(s)$. It follows from Lemma~\ref{5cyclest}(g) and Remark~\ref{rem2} that $u\notin (\cup_{i\neq 4}R_{i})\bigcup(\cup_{1\leq i\leq 5}R_{i,x})$. Notice that $u\notin Y_1\cup Y_2$, otherwise either $suv_4xv_2=P_5$ ($u\in Y_1$) or $suv_4xv_1=P_5$ ($u\in Y_2$).

Now suppose that $u\in P^{i}$ for $i\neq 2,5$. If $i=1$, then $suv_1xv_4=P_5$. If $i=3$, then $suv_4xv_1=P_5$. If $i=4$, then $suv_4xv_2=P_5$. Thus $u\notin P^1\cup P^3\cup P^4$.

Recall that $X=\{z | N_{\Sigma}(z)=x\}$. Suppose that $u\in X$. Then $suxv_1v_5=P_5$, which leads to a contradiction. Hence this completes the proof of (i).

{\it (ii):} Let $s_1s_2\in E(G[S])$ and $u\in Y_4$. Suppose that $us_1\in E(G)$ and $us_2\notin E(G)$. Then $s_2s_1uv_2x=P_5$, a contradiction. Thus $u$ is either complete or anticomplete to $S$. Now let $u\in Y_{i,x}$ for some $i\neq 4$. If $s_1u\in E(G)$ and $s_2u\notin E(G)$, then $s_2s_1uv_{i-2}v_{i-1}=P_5$, which leads to a contradiction. Thus this proves (ii.a).

Now suppose to the contrary that $S$ is adjacent to  two of $Y_4$ and $Y_{i, x}$ ($i\neq 4$). First suppose that $S$ is adjacent to   $Y_4$ and $ Y_{i,x}$ for some $i\neq 4$. By ({\it ii.a}), let $u\in Y_4$ be a vertex complete to $S$, and $v\in Y_{i,x}$ be a vertex complete to $S$. Then $uv\notin E(G)$. Otherwise, either $\{u,v,s_1,s_2,v_i\}$ induces an HVN ($i=1,2$), or $\{u, v,s_1,s_2,v_{i+2}\}$ induces an HVN ($i=5$), or $\{u,v,s_1,s_2, v_{i-2}\}$ induces an HVN ($i=3$).
This implies that either $\{s_1,u,v_{i-1},v_{i-2},v,s_2\}$ ($i=2,5$) or $\{s_1,u,v_{i+1},v_{i+2},v,s_2\}$ ($i=1,3$) induces a $T$-5-wheel, which leads to a contradiction.

Thus we may assume that $S$ is adjacent to  $Y_{i,x}$ and $Y_{j,x}$ for $i\neq j$. Let $u\in Y_{i,x}$ be a vertex complete to $S$, and $v\in Y_{j,x}$ be a vertex complete to $S$. Then $uv\notin E(G)$, otherwise either $\{u,v,s_1,s_2,v_{i-2}\}$  (when $j=i+1$) or $\{u,v,s_1,s_2,v_{i+2}\}$ (when $j=i+2$) induces an HVN. If $uv\notin E(G)$, then either $\{s_1,u,v_i,v_{i-1}, v,s_2\}$ ($j=i+1$) or $\{s_1,u,v_{i-2}, v_{i-1}, v,s_2\}$ ($j=i+2$) induces a $T$-5-wheel, which implies a contradiction. Hence $S$ is adjacent to  at most one of $Y_4$ and $Y_{i, x}$ ($i\neq 4$), which completes the proof of ({\it ii.b}).
\end{proof}

\begin{guess}\label{y5}
Let $G$ be a $(P_{5}, \textit{HVN})$-free graph. Suppose that $G$ contains an induced $Y$-5-wheel $\Sigma$ and contains no induced $T$-5-wheel. Then $G$ is $(\omega(G)+3)$-colorable.
\end{guess}

\begin{proof}
First notice that since $\Sigma$ contains an induced paw, and $S$ is not adjacent to  $\Sigma$, it follows from Lemma~\ref{homogeneous} that $S$ is $\omega(G)$-colorable. Without loss of generality, we may assume that $G[S]$ is connected.

Since all vertices in $N_{G}(x)$ are adjacent to  $x$, we have that $G[N_{G}(x)]$ is paw-free. By Theorem~\ref{colorpawfree} this implies that $N_{G}(x)$ is $\omega(G)$-colorable. Let $c_1, \ldots, c_\omega$ be the colors used on $N_{G}(x)$. To prove this theorem, we first claim that
\begin{equation}\label{eq1}
\textit{If $|S|\geq 2$, then the vertices of $N_{G}(x)$ which are adjacent to  $S$ receive at most two colors. }
\end{equation}
By Lemma~\ref{Y5S}({\it i}), $N_{G}(x)\cap N_{V(G)\backslash S}(S)\subseteq (\cup_{i\neq 4} Y_{i,x})\bigcup P^{3,x}\bigcup P^{4,x}$. It follows from Lemma~\ref{5cyclest} and Remark~\ref{rem2} that $P^{3,x}\cup P^{4,x}$ is stable, and by Lemma~\ref{Y5S}({\it ii})  $S$ is adjacent to  at most one of $Y_{i,x}$ with $i=1,2,3,5$.

Suppose that $S$ is adjacent to  some $Y_{i,x}$. If $i=1, 2$, then $Y_{i,x}\cup P^{3,x}\cup P^{4,x}$ is stable by Lemma~\ref{5cyclest}(c) and Remark~\ref{rem2}. Since $G[N_{G}(x)]$ is paw-free, by Corollary~\ref{colorpawfree}, $Y_{i,x}\cup P^{3,x}\cup P^{4,x}$ receives one color. Suppose that $i=3,5$. Let $Y_{i,x}^1$ be the set of vertices adjacent to  $P^{3, x}\cup P^{4,x}$, and $Y_{i,x}^2=Y_{i,x}\backslash Y_{i,x}^{1}$. Then $Y_{i,x}^2\cup P^{3,x}\cup P^{4,x}$ and $Y_{i,x}^1$ are both stable. Thus by Corollary~\ref{colorpawfree}, $Y_{i,x}\cup P^{3,x}\cup P^{4,x}$ receives two colors. Now suppose that $S$ is not adjacent to  $Y_{i, x}$ for any $i=1,2,3,5$. Then we can color $P^{3,x}\cup P^{4,x}$ by one color. This completes the proof of (\ref{eq1}). Let $c_1, c_2$ be the two colors used on the vertices of $N_{G}(x)$ which are adjacent to  $S$.

Now it is left to color $V(G)\backslash N_{G}(x)$. Let $D=Y_4\cup P^1\cup P^2\cup P^5\cup T$. By Lemma~\ref{5cyclest}(a,b,c), we have that $D$ is stable.

{\it \bf Case 1: $R_1\neq \emptyset$.} It follows from Lemma~\ref{Y5stable}(2a) that $R_2=\emptyset$. If $R_3\neq\emptyset$, then by Lemma~\ref{Y5stable}(2b), $R_1$ is anticomplete to $R_3$. Let $A=R_1\cup R_3\cup Y_2\cup P^3$ and $B=R_5\cup Y_1\cup P^4$. It follows from Lemma~\ref{5cyclest}(c) and Lemma~\ref{Y5stable}(1,3,4) that $A$  and  $B$ are stable.

Let $c_{\omega+1}, c_{\omega+2}$ and $c_{\omega+3}$ be three new colors used on $A, B, D$ respectively ($A\cup B$ might be empty). Recall that by Lemma~\ref{Y5S}, $S$ is not adjacent to  $A\cup B$. If $|S|\geq 2$, then by (\ref{eq1}), we may use $c_3, \ldots, c_\omega, c_{\omega+1}, c_{\omega+2}$ to color $S$. If $|S|=1$, then we just use $c_{\omega+1}$ or $c_{\omega+2}$ to color $S$. Since $V(G)=N_{G}(x)\cup A\cup B\cup D\cup S$, we obtain an $(\omega(G)+3)$-coloring of $G$.

{\it \bf Case 2: $R_2\neq \emptyset$.} By Lemma~\ref{Y5S}(2a) we have that $R_1=\emptyset$. Let $A=P^3\cup R_3\cup Y_2$ and $B=R_5\cup Y_1\cup R_2\cup P^4$. It follows from Lemma~\ref{5cyclest}(c) and Lemma~\ref{Y5stable}(1,3,4) that $A$ is stable. Similarly by Lemma~\ref{5cyclest}(c) and Lemma~\ref{Y5stable}(1,2c,3,4), we have that $B$ is stable.

Let $c_{\omega+1}, c_{\omega+2}$ and $c_{\omega+3}$ be three new colors used on $A, B, D$ respectively ($A\cup B$ might be empty). It follows from Lemma~\ref{Y5S} that $S$ is not adjacent to  $A\cup B$. When $|S|\geq 2$, by (\ref{eq1}), we may use $c_3, \ldots, c_\omega, c_{\omega+1}, c_{\omega+2}$ on $S$. If $|S|=1$, then we can use $c_{\omega+1}$ or $c_{\omega+2}$ to color $S$. Since $V(G)=N_{G}(x)\cup A\cup B\cup D\cup S$, we obtain an $(\omega(G)+3)$-coloring of $G$.

{\it\bf Case 3: $R_1\cup R_2=\emptyset$.} Now let $A=R_3\cup Y_2\cup P^3$ and $B=R_5\cup Y_1\cup P^4$. It follows from Lemma~\ref{5cyclest}(c)  and Lemma~\ref{Y5stable}(1,3,4) that $A, B$ are both stable.

Let $c_{\omega+1}, c_{\omega+2}$ and $c_{\omega+3}$ be three new colors used on $A, B, D$ respectively ($A\cup B$ might be empty). It follows from Lemma~\ref{Y5S} that $S$ is not adjacent to   $A\cup B$. If $|S|\geq 2$, we may use $c_3, \ldots, c_\omega, c_{\omega+1}, c_{\omega+2}$ to color $S$. If $|S|=1$, then we color $S$ by $c_{\omega+1}$ or $c_{\omega+2}$. Hence $G$ is $(\omega(G)+3)$-colorable as $V(G)=N_{G}(x)\cup A\cup B\cup D\cup S$. This completes the proof.
\end{proof}

\section{Proof of Theorem~\ref{main}}

Let $G$ be a $(P_5, \textit{HVN})$-free graph with no induced $T$-5-wheel or induced $Y$-5-wheel. To prove Theorem~\ref{main}, by Theorem~\ref{t5} and Theorem~\ref{y5}, it only lefts to show that  $G$ is $(\omega(G)+3)$-colorable.

\begin{guess}\label{songxu}\cite[Theorem 1.2]{arxivsong}
Let $G$ be an (odd hole, full house)-free graph. Then, $\chi(G)\leq \omega(G)+1$, and the equality holds if and only if $\omega(G)=3$ and $G$ induces an odd antihole on seven vertices.
\end{guess}

By Theorem~\ref{songxu}, we may always assume that $G$ contains induced 5-cycles.

\begin{llemma}\label{5cycle1}
Let $C=v_{1}v_{2}v_{3}v_{4}v_{5}v_1$ be an induced 5-cycle of $G$, and
\begin{itemize}
\item[](1). $S=\{u | N_{C}(u)=\emptyset\}$,
\item[](2). $R_{i}=\{u | N_{C}(u)=\{v_{i-1}, v_{i+1}\}\}$ for $i\in \{1, \ldots, 5\}$,
\item[](3). $P^{i}=\{u | N_{C}(u)=\{v_{i-1}, v_{i}, v_{i+1}, v_{i+2}\}\}$ for $i\in \{1, \ldots, 5\}$,
\item[](4). $T=\{u | N_{C}(u)=V(C)\}$.
\end{itemize}
Then $V(G)=T\bigcup S\bigcup\cup_{1\leq i\leq 5}(R_{i}\cup P^i)$. Moreover, $T\bigcup (\cup_{1\leq i\leq 5} P^i)$ and $R_i$ ($i\in \{1, \ldots, 5\}$) are stable.
\end{llemma}

\begin{proof}
The lemma directly follows from Lemma~\ref{5cycle}, Lemma~\ref{5cyclest} and the fact that $G$ contains no induced $T$-5-wheel or $Y$-5-wheel.
\end{proof}

\begin{llemma}\label{r}
Suppose that $R_{i-1}\neq\emptyset$ and $R_{i+1}\neq\emptyset$. Then $R_{i-1}$ is anticomplete to $R_{i+1}$.
\end{llemma}

\begin{proof}
Let $u\in R_{i-1}$ and $w\in R_{i+1}$. Suppose that $uw\in E(G)$. Then $\{u,v_{i-1}, v_{i+2}, v_{i-2}, v_{i},w\}$ induces a $Y$-5-wheel, which leads to a contradiction. Thus $R_{i-1}$ is anticomplete to $R_{i+1}$.
\end{proof}

Suppose that $\cup_{1\leq i\leq 5}R_i\neq \emptyset$. We may always assume that $R_1\neq\emptyset$ by relabelling the vertices on $C$. Then it follows from Lemma~\ref{r} that $R_2\cup R_5$ is stable. Here notice that if $\cup_{1\leq i\leq 5}R_i= \emptyset$, then this holds trivially. Now Let $A=R_1\bigcup R_3\bigcup(\cup_{i\neq 4}P^i)\bigcup T, B_1=R_2\cup R_5, B_2=R_4$ and $D=P^4$. Clearly $A\cup B_1\cup B_2\cup D=V(G)\backslash S$.

Since each vertex of $A$ is adjacent to   $v_2$, we have that $A$ is $\omega(G)$-colorable, and let the colors be $c_1, \ldots, c_\omega$. Let $R_1^1=\{u\in R_1 | N_{(\cup_{i\neq 4}P^i)\bigcup T}(u)\neq \emptyset\}$ and $R_1^0=R_1\backslash R_1^1$. Then by Lemma~\ref{colorpawfree}, we may use $c_1$ on $(\cup_{i\neq 4}P^i)\bigcup T\cup R_1^0$ and $c_2$ on $R_1^1$.

Let $c_{\omega+1}, c_{\omega+2}$ and $c_{\omega+3}$ be three new colors used on $B_1,B_2,D$ respectively ($B_1\cup B_2$ might be empty). It follows from Lemma~\ref{5cyclest}(g) that $S$ is not adjacent to   $B_1\cup B_2$. If $|S|=1$, then we may use $c_{\omega+1}$ or $c_{\omega+2}$ on $S$, which gives an$(\omega(G)+3)$-coloring for $G$. Thus in the rest of the proof, we may assume that $|S|\geq 2$.

Let $W=T\bigcup (\cup_{1\leq i\leq 5} P^i)$. By Lemma~\ref{5cyclest}(g), $N_{V(G)\backslash S}(S)\subseteq W$. Since $G$ is connected, we have that $W\neq\emptyset$. Let $S_1=\{s\in S | N_{W}(s)\neq\emptyset\}$ and $S_0=S\backslash S_1$. Let $w\in W$ be a vertex with the largest number of neighbours in $S_1$. We claim that
\begin{equation}\label{eq5}
\textit{$w$ is complete to $S_1$}.
\end{equation}
Suppose that $w$ is not complete to $S_1$, and let $z\in S_1$ such that $wz\notin E(G)$. Since $w$ has the maximum number of neighbours in $S_1$, we have that there must exist a vertex $y\in N_{S_1}(w)$ and a vertex $w'\in W$ such that $w'z\in E(G)$ and $w'y\notin E(G)$. This implies that either $ywv_iw'z=P_5$ or $\{z, y, w,v_i,w', v_{i-1}\}$ induces a $T$-5-wheel for some $1\leq i\leq 5$, a contradiction. This completes the proof of (\ref{eq5}). Thus
\begin{equation}\label{eq3}
\textit{$G[S_1]$ is paw-free, and $S_1$ is $\omega(G)$-colorable}.
\end{equation}

It follows from Lemma~\ref{5cyclest}(g) that $S_1$ is not adjacent to   $B_1\cup B_2$. If $|S_1|=1$, then we can use $c_{\omega+1}$ or $c_{\omega+2}$ on $S_1$. If $|S_1|\geq 2$, we may use $c_3, \ldots, c_\omega, c_{\omega+1}, c_{\omega+2}$ to color $S_1$.  Hence $V(G)\backslash S_0$ is $(\omega(G)+3)$-colorable. Now it is left to extend this $(\omega(G)+3)$-coloring to $S_0$ for completing the proof of Theorem~\ref{main}.

Recall that $S_0$ is the set of vertices of $S$ which are anticomplete to  $W$. If $|S_0|=1$, then we can just color $S_0$ by $c_{\omega+3}$. Now we may assume that $|S_0|\geq 2$. Since $W\neq\emptyset$, we have that $G[V(C)\cup W]$ contains an induced paw. By Lemma~\ref{homogeneous}, $S_0$ is homogeneous and $\omega(G)$-colorable.

By (\ref{eq3}), $G[S_1]$ is either complete multipartite, or triangle-free. First suppose that $G[S_1]$ is complete multipartite, and let the partitions be $S_{1,1}, \ldots, S_{1,t}$. Observe that we can always color $S_1$ such that each $S_{1,i}$ receives exactly one color. We claim that
\begin{equation}\label{eq4}
\textit{$S_0$ is adjacent to   at most one partite set of $S_1$}.
\end{equation}
Suppose to the contrary that $S_0$ is adjacent to   $S_{1,1}$ and $S_{1,2}$, that is, there exist $u_1\in S_{1,1}, u_2\in S_{1, 2}$ such that $u_1, u_2$ are both complete to $S_0$. Let $a,b\in S_0$ where $ab\in E(G)$. It follows from (\ref{eq5}) that there exists $w\in W$ which is complete to $S_1$. Thus $\{w, u_1, u_2,a,b\}$ induces an HVN as $wa, wb\notin E(G)$. Hence this proves (\ref{eq4}).

Let $S_{1,1}$ be the partite set adjacent to   $S_0$, and suppose it is colored by $c_3$. Then we can color $S_0$ by $c_4, \ldots, c_{\omega+2}, c_{\omega+3}$. Hence we obtain an $(\omega(G)+3)$-coloring for $G$.

Finally, suppose that $G[S_1]$ is triangle-free. Then $G[S_1]$ is $3$-colorable, and we may assume that it is colored by $c_3, c_4, c_5$. Since $S_0$ is anticomplete to $V(G)\backslash S$, we can color $S_0$ by $c_1, c_2, c_6, \ldots, c_{\omega+2}, c_{\omega+3}$. Therefore, $G$ is $(\omega(G)+3)$-colorable, which completes the proof of Theorem~\ref{main}.

\bibliographystyle{acm}
\bibliography{history}

\end{document}